\documentclass[arxiv,proc,reqno]{ml-gen}

\usepackage{tikz}   
\usepackage{amsmath,upref} 
\usepackage{amssymb,amsfonts,mathrsfs}

\usepackage[symbol]{footmisc}
\usepackage{perpage}
\MakePerPage[5]{footnote}

\usepackage[colorlinks=true,linkcolor=blue,citecolor=blue,backref=page]{hyperref}
\usepackage{math60,mlthm}
\usepackage{local}

\usepackage{euler}

\date{\today}

\begin{document}

\title[Regularizing infinite sums of zeta-determinants]
{Regularizing infinite sums of zeta-determinants}

\author{Matthias Lesch}
\address{Mathematisches Institut,
Universit\"at Bonn,
53115 Bonn,
Germany}

\email{ml@matthiaslesch.de, lesch@math.uni-bonn.de}
\urladdr{www.matthiaslesch.de, www.math.uni-bonn.de/people/lesch}

\author{Boris Vertman}
\address{Mathematisches Institut,
Universit\"at Bonn,
53115 Bonn,
Germany}
\email{vertman@math.uni-bonn.de}
\urladdr{www.math.uni-bonn.de/people/vertman}

\thanks{Both authors were supported by the 
        Hausdorff Center for Mathematics.}

\date{This document was compiled on: \today}

\begin{abstract}
We present a new multiparameter resolvent trace expansion for elliptic
operators, polyhomogeneous in both the resolvent and auxiliary variables.
For elliptic operators on closed manifolds the expansion is a simple
consequence of the parameter dependent pseudodifferential calculus. As
an additional nontrivial toy example we treat here Sturm-Liouville operators with
separated boundary conditions.

As an application we give a new formula, in terms of regularized sums, for the
$\zeta$--determinant of an infinite direct sum of Sturm-Liouville operators.
The Laplace-Beltrami operator on a surface of revolution decomposes into an
infinite direct sum of Sturm-Louville operators, parametrized by the
eigenvalues of the Laplacian on the cross-section $\mathbb{S}^1$.  We apply
the polyhomogeneous expansion to equate the zeta-determinant of the
Laplace-Beltrami operator as a regularized sum of zeta-determinants of the
Sturm-Liouville operators plus a locally computable term from the
polyhomogeneous resolvent trace asymptotics.  This approach provides a
completely new method for summing up zeta-functions of  operators
and computing the meromorphic extension of that infinite sum to $s=0$. 

We expect our method to extend to a much larger class of operators.
\end{abstract}

\maketitle
\tableofcontents

\section{Introduction and formulation of the result}
Various geometric problems involve zeta-determinants of Hodge-Laplace
operators which decompose into an infinite sum of scalar Laplace-type
operators. The most prominent example seems to be the discussion of analytic
torsion on spaces with conical singularities, where the problem of computing
the zeta-determinant of an infinite sum of scalar operators arises naturally
and has motivated the work of the first author in \cite{Les:DOR}. 

The basic approach to this problem is given by summing up zeta-functions 
$\zeta(s,\Delta_\lambda), \lambda \in \N_0,$ of the scalar Laplace-type operators 
$\Delta_\lambda$ for $\Re(s) \gg 0$ and computing the
meromorphic extension of that infinite sum to $s=0$. This approach was taken by Spreafico 
in \cite{Spr:ZFA, Spr:ZIF}, where the intricate task of constructing a meromorphic extension 
is addressed for bounded cones. Compare also the discussion by Bordag, Kirsten and 
Dowker in \cite{BKD:HKA} and by the second author in \cite{Ver:ATO}.

In this article we present a conceptually new method for computing the zeta determinant 
of an infinite sum of operators, which uses a new polyhomogeneous 
resolvent trace expansion. Our model setup here is a surface of 
revolution. The spectral decomposition on $\mathbb{S}^1$ decomposes the Laplace-Beltrami operator 
$\Delta$ on a surface of revolution into an infinite sum of 
Sturm-Liouville operators $\Delta_\lambda, \lambda \in \N_0,$ on a finite
interval with separated boundary conditions.

We establish an expansion of the resolvent trace for $\Delta_\lambda$, 
polyhomogeneous both in $\lambda$ and the resolvent parameter, and prove that 
the zeta-determinant of $\Delta$ is given by a regularized sum of zeta-determinants for 
$\Delta_\lambda, \lambda \in \N_0$. 
Moreover, the polyhomogeneous resolvent trace expansion explains 
the origin of the trace coefficients in the expansion of 
$\Tr(\Delta+z^2)^{-2}$ as $z\to \infty$, which do not appear in the 
corresponding (standard) resolvent expansions of the scalar operators $\Delta_\lambda$. 

\subsection{Laplace-Beltrami operator on a surface of revolution}
Let $(M=[0,1]\times \mathbb{S}^1, g=dx^2\oplus r(x)^2 g_{\mathbb{S}^1})$ be a
surface of revolution with $r\in C^{\infty}[0,1], r > 0$. The metric is a warped
product and the associated Laplace-Beltrami operator is given by the
differential expression 
\begin{equation}
\Delta = -\frac{\partial^2}{\partial x^2} - 
\frac{r'(x)}{r(x)} \frac{\partial}{\partial x} + \frac{1}{r(x)^2} \Delta_{\mathbb{S}^1},
\end{equation}
acting on $C^{\infty}_0((0,1) \times \mathbb{S}^1)$, the 
space of complex-valued smooth compactly supported functions on $(0,1)\times 
\mathbb{S}^1$. The natural $L^2$-space with respect to the metric $g$ 
is $L^2(M,g)=L^2([0,1]\times S^1, r(x)\, dx\, \textup{dvol}(g_{\mathbb{S}^1}))$.
Under the unitary map
\begin{equation}
\label{unitary}
\Phi: L^2(M,g) \to L^2([0,1], L^2(\mathbb{S}^1,g_{\mathbb{S}^1})), 
\ (\Phi u)(x):= u(x)\sqrt{r(x)},
\end{equation}
the Laplacian $\Delta$ transforms into the operator
\begin{equation}
\Phi \Delta \Phi^{-1} = -\frac{\partial^2}{\partial x^2} + \frac{1}{r(x)^2} \Delta_{\mathbb{S}^1} 
+ \left[\frac{r''(x)}{2r(x)} - \left(\frac{r'(x)}{2r(x)}\right)^2\right], 
\end{equation}
acting in $L^2([0,1],L^2(\mathbb{S}^1))$.
The functions $\left(\frac{1}{\sqrt{2\pi}}e^{i\lambda x}\right)_{\lambda \in \Z}$
form an orthonormal basis of $L^2(\mathbb{S}^1)$ of eigenfunctions 
of $\Delta_{\mathbb{S}^1}$ to the eigenvalues $\lambda^2, \lambda \in \Z$. 
The eigenvalues $\lambda^2\neq 0$ have multiplicity two, the eigenvalue $\lambda^2=0$ has multiplicity one. 
Hence we have a decomposition

\begin{equation}
\label{laplace-decomp}
\begin{split}
\Phi \Delta \Phi^{-1} &= -\frac{\partial^2}{\partial x^2} + \frac{1}{r(x)^2} \Delta_{\mathbb{S}^1} 
+ \left[\frac{r''(x)}{2r(x)} - \left(\frac{r'(x)}{2r(x)}\right)^2\right] \\
&=\bigoplus_{\lambda = -\infty}^{\infty} \left( -\frac{\partial^2}{\partial x^2} + \frac{\lambda^2}{r(x)^2} 
+ \left[\frac{r''(x)}{2r(x)} - \left(\frac{r'(x)}{2r(x)}\right)^2\right]\right)
=: \bigoplus_{\lambda=0}^{\infty} \Delta_\lambda,
\end{split}
\end{equation}
into a direct sum of one-dimensional Sturm-Liouville type operators.
We consider separated Dirichlet or generalized Neumann boundary conditions for 
$\Delta$. It is straightforward to check that under the unitary transformation 
$\Phi$ they correspond to separated Dirichlet or generalized Neumann 
boundary conditions for $\Phi \Delta \Phi^{-1}$ and that the 
resulting self-adjoint operator is compatible with the 
decomposition \Eqref{laplace-decomp}. By slight abuse of notation
we denote the transformed self-adjoint operator again by $\Delta$.
Accordingly, the resulting self-adjoint extensions of $\Delta_\lambda, \lambda\in \Z$,
are again denoted by $\Delta_\lambda$. So the operators are 
identified with their self-adjoint extensions which does not lead to 
notational confusion as the boundary conditions are fixed.

\subsection{Hadamard partie finie regularized sums and integrals} We briefly
recall some facts about regularized limits and the Hadamard partie finie
regularization of integrals, for more details cf. \cite[Sec.~2.1]{Les:OFT}.
Furthermore, we introduce a regularized sum based on the Hadamard partie finie
and the Euler MacLaurin summation formula. Of course this idea is not new, cf.
e.g. \cite{ChaCon:SAR}, \cite{GSW:EFS}.

We write $\R_+=[0,\infty)$ and $\R^*_+=(0,\infty)$. 
Let $f:\R^*_+\to \C$ be a function with a (partial) asymptotic expansion 
\begin{equation}
\label{reg-limit}
f(x) \sim \sum_{j=1}^{N-1} \sum_{k=0}^{M_j} a_{jk}x^{\A_j} \log^k(x) +
             \sum_{k=0}^{M_0} a_{0k} \log^k(x) +  f_N(x), \ \quad \ x\ge x_0>0,
\end{equation}
where $\ga_j\in\C$ are ordered with decreasing real part
and the remainder $f_N(x)=o(1)$ (Landau notation) as $x\to \infty$. 
We define its \emph{regularized limit} for $x\to \infty$ as 
\begin{equation}
\LIM_{x\to \infty}f(x) :=a_{00}.
\end{equation}
If $f$ has an expansion of the form \Eqref{reg-limit} as $x\to 0$ then
the regularized limit as $x\to 0$ is defined accordingly.

If $f$ is locally integrable and for $N\in \N$ sufficiently large, the remainder 
$f_N\in L^1[1,\infty)$, the integral $\int_1^R f(x)dx$ also admits an asymptotic 
expansion of the form \Eqref{reg-limit} and we can define the regularized integral 
as 
\begin{equation}\label{EqRegInt1}
\regint_1^\infty f(x)dx := \LIM_{R\to \infty}\int_1^R f(x) dx.
\end{equation}
Similarly, 
$\regint_0^1 f(x)dx := \LIM\limits_{\eps\to 0}\int_\eps^1 f(x) dx,$
if this regularized limit exists. Functions $f$ for which $\regint_0^\infty
f:= \regint_0^1 f+ \regint_1^\infty$ exists will be called
$\regint$ --integrable. As an example we mention the formulas
\begin{align}
 \regint_z^\infty x^{\ga+1} dx
    &= \begin{cases} -\frac{z^{\ga+2}}{\ga+2}, & \ga+2\not=0, z>0,\\
                     -\log z                 , & \ga=-2, z>0,\\
                      0,                        & z=0.
                     \end{cases}\label{EqExRegInt1}\\
 \regint_z^\infty x^{\ga+1}\log x  dx
    &= \begin{cases}  \frac{z^{\ga+2}}{(\ga+2)^2} - \frac{z^{\ga+2}\log z}{\ga+2}, & \ga+2\not=0, z>0,\\
                     -\frac 12 \log^2 z                 , & \ga=-2, z>0,\\
                      0,                        & z=0.
                     \end{cases}\label{EqExRegInt2}
 \end{align}                     

The regularized integral has a peculiar change of
variables rule.

\begin{lemma}[{\cite[Lemma 2.1.4]{Les:OFT}}]\label{LChangeVar} Let $f:\R_+\to\C$ be 
 $\regint$ --integrable and denote by $A_\infty$ (resp. $A_0$) the coefficients
 of $x^{-1}$ in the expansion \Eqref{reg-limit} as $x\to \infty$ (resp. $x\to
 0$) and assume for simplicity that there are no terms of the form $x\ii
 \log^k x,$ with $k\ge 1$, in these expansions. Then  for $\gl >0$
 \begin{equation}\label{EqChangeVar}
  \regint_0^\infty f(x) dx = \gl \regint_0^\infty f(\gl \cdot x) dx -
  A_\infty \log \gl + A_0 \log\gl.
 \end{equation}
\end{lemma}

We also need a notion of a partie finie regularized sum. 

\begin{prop} \label{PRegSum}
Let $f:[1,\infty)\to \C$ be a function with an asymptotic expansion
\Eqref{reg-limit} as $x\to\infty$ where $f_N(x)=O(x^{-1-\delta}), x\to\infty$
for some $\delta>0$. Then 
\begin{equation} \label{EqRegSum1}
\regsum_{\lambda=1}^\infty f(\lambda) := \LIM_{N\to \infty}
\sum_{\lambda=1}^{N} f(\lambda)
\end{equation}
exists. Moreover, if $f$ is smooth and if the asymptotic expansion 
may be differentiated $(2M+1)$ times, $2M>\Re(\A_1)$, then
\begin{align}
 \regsum_{\lambda=1}^\infty f(\lambda)  =& \regint_1^\infty f(x)dx + \sum_{k=1}^M \frac{B_{2k}}{(2k)!} 
    \Bigl( \LIM_{N\to \infty} f^{(2k-1)}(N) -
    f^{(2k-1)}(1)\Bigr) \label{EqRegsum2}\\ 
     &+\frac{1}{2}f(1) +\frac 12
              \LIM_{N\to\infty} f(N) + \frac{1}{(2M+1)!} \int_1^\infty 
              B_{2M+1}(x-[x])f^{(2M+1)}(x)dx . \nonumber 
\end{align}
Here $B_j$ denotes the $j$-th Bernoulli number, $B_j(x)$ the $j$-th Bernoulli 
polynomial, and $f^{(j)}$ denotes the $j$-th derivative of $f$. 
\end{prop}
\begin{proof} Let us first assume that $f$ is smooth with an asymptotic
 expansion which may be differentiated. Then the Euler MacLaurin summation formula 
yields for $N\in \N$
\begin{equation}
\label{EM}
\begin{split}
\sum_{\lambda=1}^{N} f(\lambda) 
    =& \int_1^N f(x)dx +\sum_{k=1}^M \frac{B_{2k}}{(2k)!} 
             \Bigl(f^{(2k-1)}(N)-f^{(2k-1)}(1)\Bigr) \\ 
     &\quad + \frac{1}{(2M+1)!} \int_1^N 
              B_{2M+1}(x-[x])f^{(2M+1)}(x)dx  +\frac{1}{2}(f(1) + f(N)).
\end{split}
\end{equation}
This shows immediately that  $\sum_{\gl=1}^N f(\gl)$ admits an asymptotic 
expansion of the form \Eqref{reg-limit} and that the regularized limit
is given by the right hand side of \Eqref{EM}.

For the general $f$ as in the proposition we only have to note that 
$\sum_{\gl=1}^\infty f_N(\gl)$ converges absolutely and that the first
part of the proof applies to the individual summands
$x^{\ga_j} \log^k x$ of the expansion \Eqref{reg-limit}.
\end{proof}

\subsection{Statement of the main results}
Our first main result establishes a Fubini-type theorem for regularized integrals and
is one fundamental ingredient in the derivation of our main Theorem
\plref{main-thm} below.
\begin{theorem}[Fubini Theorem for regularized sums and integrals] \label{fubini}
Assume $f\in C^{\infty}(\R_+^2)$ is of the form 
\begin{align}
f(x,y) = \sum_{j=0}^{N-1}f_{\A_j}(x,y) + F_N(x,y),
\end{align}
where each $f_{\A_j}\in C^\infty(\R^2_+\setminus \{(0,0)\})$ is homogeneous of order $\A_j\in \C$
and the remainder $F_N= O((x^2+y^2)^{-1/2-\delta}), x^2+y^2\ge r_0>0$ for some $\delta>0$. 
Then, for $a,b\ge 0, a+b>0$
\begin{align}\label{fubini-eq}
\regint_a^\infty\regint_b^\infty f(x,y) \, dy \, dx = 
\regint_b^\infty \regint_a^\infty f(x,y) \, dx \, dy - 
\int_0^\infty f_{-2}(x,1)\log x\,  dx.
\end{align}  
and for $a\ge 0, \gl_0\ge 1$
\begin{equation}\label{EqFubiniSum}
\begin{split}
 &\regint_a^\infty\regsum_{\gl=\gl_0}^\infty f(x,\gl) \, dx = 
 \regsum_{\gl=\gl_0}^\infty \regint_a^\infty f(x,\gl) \, dx -
\int_0^\infty f_{-2}(x,1)\log x\, dx \\
&-\frac{1}{2} \regint_0^\infty f_{-1} (x,1) dx - 
\sum_{k=1}^M \frac{B_{2k}}{(2k)!} \regint_0^\infty \partial_2^{2k-1} f_{2k-2} (x,1) dx
\end{split}
\end{equation}  
\end{theorem}

Note that the integration in the
correction term on the right is from $0$ to $\infty$ independently
of the values of $a,b, \gl_0$. The integral 
$\int_0^\infty f_{-2}(x,1) \log x \, dx$ exists in the ordinary sense
since $f_{-2}(x,1)$ is smooth up to $x=0$ and $O(x^{-2})$ as $x\to\infty$.

Our second main result addresses the polyhomogeneous asymptotic expansion of the 
resolvent trace for $(\Delta_\lambda+z^2)^{-1}$ \emph{jointly} in $(\lambda,z)\in \R_+^2$. 

\begin{prop}\label{phg-trace}
Let $\lambda\in \R$ and $V,W \in C^\infty(\R)$ with $V(x)>0$ for all $x\in\R$.
Consider the differential operator
\begin{align}\label{EqDeltaLa}
\Delta_{\lambda,0}=-\frac{\partial^2}{\partial x^2} + \lambda^2 V + W
: C^{\infty}_0(0,1) \to C^{\infty}_0(0,1).
\end{align} 
Let $\Delta_\lambda$ be the self-adjoint extension of $\Delta_{\lambda,0}$, obtained
by imposing \emph{separated} Dirichlet or generalized Neumann boundary conditions. 
Then the resolvent $(\Delta_\lambda+z^2)^{-1}$ 
is trace class for $|(\gl,z)|\geq z_0$ large, 
and its trace admits the following polyhomogeneous expansion 
\begin{align}\label{phg-trace-eq}
\partial_\lambda^{\A}\partial_z^{\beta}\Tr(\Delta_\lambda+z^2)^{-1} \sim \sum_{i=0}^{\infty} 
h_i (\lambda,z), \quad
|(\lambda,z)| \rightarrow \infty,
\end{align}
where each $h_i\in C^{\infty}(\R^2_+\setminus \{(0,0)\})$ is homogeneous of order $(-\gamma_i)$, 
$\gamma_i:=i+1+\A+\beta$. Note that $h_i$ depends on $\A,\beta$. Moreover, the leading term 
$h_0$ comes from the interior expansion only.
\end{prop}

In particular
\begin{align}
\label{trace-2-expansion}
\Tr(\Delta_\lambda+z^2)^{-2} = - (2z)^{-1} \partial_z
\Tr(\Delta_\lambda+z^2)^{-1} \sim \sum_{i=0}^\infty h_i (\lambda,z), \ |(\lambda,z)| \to \infty,
\end{align}
where each $h_i\in C^{\infty}(\R^2_+\setminus \{(0,0)\})$
is homogeneous of order $(-\gamma_i)$, $\gamma_i:=i+3$,  jointly in both variables. 

\begin{remark}\label{boutet-de-monvel}
There are several fundamental approaches to the analysis on manifolds
with boundaries or singularities, among them those which can be traced
back to Kondratiev's work on conical singularities (cf. e.g.
Egorov-Schulze \cite{EgSh:PDO}), those based on Boutet de Monvel's
formalism (e.g. Grubb \cite{Gru:FCO}) as well as those going back to
the $b$-calculus of Melrose (Melrose \cite{Mel:APS}).
Consequently there exist various principal ways to establishing
Proposition \ref{phg-trace}.  On the one hand, the operator
$(\Delta_\lambda + z^2)$ is a parameter-elliptic element in the Boutet
de Monvel calculus for boundary value problems with parameter
$(\lambda, z)$. Hence its resolvent admits a polyhomogeneous
asymptotic expansion in the two parameters, resulting in the statement
of \eqref{phg-trace-eq}.  In this context we should also mention the
central contributions by Seeley \cite{See:TRO}. 

On the other hand, the Schwartz kernel of the resolvent $(\Delta_\lambda + z^2)^{-1}$ is a sum of 
an interior parametrix plus a polyhomogeneous function on some $b$-blowup space up to some remainder 
of higher order. This allows for a derivation of \eqref{phg-trace-eq} using elements of Melrose's $b$-calculus.
While we are not attempting to compare both ansatzes, we decided to present the second approach here, 
which lays out a framework for future analysis of related questions in case of singular operators.  
\end{remark}

Proposition \ref{phg-trace} implies in particular the well-known fact, that for fixed 
$\lambda$ there is an asymptotic expansion as $z\to \infty$
\begin{equation}\label{1.14a}
\Tr(\Delta_\lambda+z^2)^{-1} \sim \sum_{k=0}^\infty b_k z^{-k-1}, 
\end{equation}
which may be differentiated in $z$, e.g.,
\begin{equation}
 \begin{split}
\label{1.14b}
\Tr\bl \Delta_\lambda+z^2\br^{-2} = -\frac{1}{2z} 
                \frac{d}{dz} \Tr\bl\Delta_\gl + z^2\br^{-1}
                \sim \sum_{k=0}^\infty \frac{k+1}{2} b_k z^{-k-3}
                =: \sum_{k=0}^\infty c_k z^{-k-3}.
               \end{split}
              \end{equation}
The leading orders in the resolvent trace asymptotics for $\Delta:=\oplus \Delta_\lambda, \lambda \in \Z$, 
and $\Delta_\lambda$ are fundamentally different. On the one hand $\Tr(\Delta_\lambda+z^2)^{-2}=O(z^{-3})$ 
whereas $\Tr(\Delta+z^2)^{-2}=O(z^{-2})$ as $z\to \infty$. Indeed, the resolvent trace asymptotics of $\Delta_\lambda$ 
does not sum up to the asymptotics of the full resolvent trace for $\Delta$
in an obvious way. Nevertheless we have the

\begin{theorem}\label{trace-sum}
In the notation of Proposition \textup{\ref{phg-trace}} we have
for the operator 
$ \Delta:=\bigoplus\limits_{\gl=-\infty}^\infty \Delta_\gl$
the resolvent trace expansion
\begin{align}\label{1.15}
\Tr(\Delta+z^2)^{-2} = \sum_{\lambda=-\infty}^{\infty}\Tr(\Delta_\lambda+z^2)^{-2} 
\sim \sum_{k=2}^\infty a_k z^{-k}, \ z \to \infty.
\end{align}
\end{theorem}
Note that $\Delta$ is just an abstract sum of operators of the form
\Eqref{EqDeltaLa} and therefore does not necessarily have an interpretation
as a realization of an elliptic boundary value problem on a surface.
If, like in the case of a surface of revolution, $\Delta$ is a realization of a 
local elliptic boundary value problem, then Theorem \ref{trace-sum} is well-known, 
e.g. \cite[Sec. 1.11]{Gil:ITH}.
More important than the result itself, however, is our method of proof using
the polyhomogeneous resolvent trace expansion in Proposition \ref{phg-trace}
and \Eqref{EM}, which explains precisely the difference in the leading
orders of the resolvent trace expansion of
$\Tr(\Delta_\lambda+z^2)^{-2}=O(z^{-3}), z\to\infty,$ 
and their sum $\Tr(\Delta+z^2)^{-2}=O(z^{-2}), z\to\infty$.

We now define the associated zeta-regularized determinants, following \cite[(1.7)]{Les:DOR}.
The zeta-function of $\Delta_\lambda$ is defined for $\Re(s)\gg 0$ by 
\begin{align}\label{zeta}
\zeta (s, \Delta_\lambda) = \sum_{\mu \in \textup{Spec}\Delta_\lambda
\setminus \{0\}}
\textup{m($\mu$)} \mu^{-s}, 
\end{align}
where $\textup{m($\mu$)}$ denotes the multiplicity of the eigenvalue $\mu >0$. Using
the identity 
\begin{align}\label{1.17}
\zeta (s, \Delta_\lambda) = 2 \, \frac{\sin \pi s}{\pi} 
\regint_0^\infty z^{1-2s} \Tr(\Delta_\lambda+z^2)^{-1} dz,
\end{align}
the asymptotics \Eqref{1.14a} implies that $\zeta (s, \Delta_\lambda)$ extends meromorphically 
to $\C$ with $s=0$ being a regular point. From \Eqref{1.14a} and \Eqref{1.17}
one derives the 
formula for $\log \det_{\zeta} \Delta_\lambda = -\zeta' (0, \Delta_\lambda)$
\begin{equation}
\log \det\nolimits_{\zeta} \Delta_\lambda = -2 \regint_0^\infty z \Tr(\Delta_\lambda + z^2)^{-1} dz.
\end{equation}
The resolvent $(\Delta+z^2)^{-1}$ is not trace class and we cannot employ exactly the same
formulas for the definition of $\det_{\zeta} \Delta$. However, integration by parts in \Eqref{1.17} yields
\begin{align}
\label{1.18}
\zeta (s, \Delta_\lambda) = 2 \, \frac{\sin \pi s}{\pi (1-s)} 
\regint_0^\infty z^{3-2s} \Tr(\Delta_\lambda+z^2)^{-2} dz,
\end{align}
and thus with \Eqref{1.14b}
\begin{equation}\label{EqLogDetzLa}
 \log \det\nolimits_{\zeta} \Delta_\lambda = 
-2 \regint_0^\infty z^3 \Tr(\Delta_\lambda + z^2)^{-2} dz.
\end{equation}
Invoking Theorem \ref{trace-sum} one sees that \Eqref{1.18} is still valid 
for $\Delta$ instead of $\Delta_\lambda$. Moreover, the asymptotic expansion \Eqref{1.15}
implies that \Eqref{EqLogDetzLa} also holds for $\Delta$.

Note that unlike in the standard convention, here we do not set the zeta-determinant to zero 
for operators that are not invertible. Our third and final main result now reads as follows.

\begin{theorem} \label{main-thm} In the notation of Proposition \plref{phg-trace} and Theorem
 \plref{trace-sum} we have for the zeta-regularized sum of $\Delta$
\begin{equation}
\begin{split}
&\log \det\nolimits_{\zeta} \Delta = \regsumwide_{\lambda=-\infty}^\infty \log \det\nolimits_{\zeta} \Delta_\lambda 
- 4 \int_0^\infty z^3 h_{2}(1,z) \log (z) dz \\
& + 2 \regint_0^\infty z^3 h_{1}(1,z) dz + 2 B_2 \regint_0^\infty z^3 \partial_{\lambda} h_{0}(1,z) dz. 
\end{split}
\end{equation}
where $h_{j}, j=0,1,2$ denotes the homogeneous term of degree $(-3-j)$ in the 
polyhomogeneous asymptotic expansion of $\Tr(\Delta_\lambda + z^2)^{-2}$
as $|(\lambda, z)|\to \infty$, respectively. 
\end{theorem}
Needless to say
\begin{align}
\regsumwide_{\lambda=-\infty}^\infty f(\lambda) := 
\LIM_{N\to \infty} \sum_{\lambda=-N}^{N} f(\lambda)
= \LIM_{N\to \infty} \Bigl(\sum_{\lambda=1}^{N} f(\lambda) +
\sum_{\lambda=1}^{N} f(-\lambda) + f(0)\Bigr).
\end{align}

In the diploma thesis of B. Sauer \cite{Sau:ORT} the term $\int_0^\infty z^3 h_{2}(1,z) \log (z) dz$
has been identified in terms of $V,W$ and their derivatives at the boundary.

Note that by \Eqref{trace-2-expansion}, the correction terms $h_{0,1,2}$ in Theorem \ref{main-thm}
are the leading three (local) components of the polyhomogeneous asymptotic expansion of 
$\Tr(\Delta_\lambda + z^2)^{-2}$.

\section{Polyhomogeneous expansion of the resolvent trace}
In this section we establish a polyhomogeneous asymptotic expansion of the 
resolvent trace for $(\Delta_\lambda+z^2)^{-1}$ jointly in $(\lambda,z)\in \R_+^2$. 
The discussion is separated into two parts for the interior and the boundary parametrices. 
We begin with the interior parametrix where the polyhomogeneous expansion is 
a consequence of the strongly parametric elliptic calculus. 

\subsection{The interior parametrix}
We will use here freely the calculus of pseudo-differential operators 
with parameter, for a survey type exposition see \cite[Sec. 4 and 5]{Les:PDO}.

Consider the differential operators 
\begin{equation}
\begin{split}
\Delta_{\lambda,0}&=-\partial_x^2 + \lambda^2 V +W: C^\infty _0(0,1) \to C^\infty _0(0,1),\\
\Delta^\R_\lambda &=-\partial_x^2 + \lambda^2 V +W: C^\infty _0(\R) \to C^\infty _0(\R),
\end{split}
\end{equation}
where $V,W \in C^\infty(\R)$ with $V>0$. As before, $\Delta_\lambda$ is a self-adjoint 
extension of $\Delta_{\lambda,0}$ in $L^2[0,1]$, obtained by imposing separated 
Dirichlet or generalized Neumann boundary conditions. The boundary conditions will be specified 
in the next section. We write
\begin{equation}
\Delta (\lambda, z) :=\Delta_\lambda + z^2,\quad
\Delta^\R (\lambda, z):=\Delta^\R_\lambda +z^2.
\end{equation}
Then $\Delta^\R(\lambda,z)$ is elliptic in the parametric sense with parameter $(\lambda,z)$ in the cone 
$\Gamma = \R^+_{\lambda} \times \R^+_{z}$. The space of classical 
parameter dependent pseudo-differential operators of order $m$ is, as usual,
denoted by $\CL^{m}(\R;\Gamma)$. By \cite[Sec. II.9]{Shu:POS} 
$\Delta^\R(\lambda,z)$ admits a parametrix $R\equiv R(\lambda,z)\in \CL^{-2}(\R;\Gamma)$,
such that
$$\Delta^\R(\lambda,z)R-I, \ R\Delta^\R(\lambda,z)-I \in \CL^{-\infty}(\R;\Gamma).$$

Since $\textup{ord}R + \dim \R=-1<0$ the Schwartz kernel 
$k(\cdot, \cdot; \lambda, z)$ of $\Delta^\R(\lambda,z)^{-1}$
is a continuous function and on the diagonal it has an asymptotic expansion
\begin{equation}
\label{interior}
k(x,x;\lambda,z)
\sim \sum_{j=0}^\infty e_j \!\left(\frac{(\lambda,z)}{|(\lambda,z)|}\right)
|(\lambda,z)|^{-1-j}, \ |(\lambda,z)| \to \infty, \ (\lambda,z)\in \Gamma,
\end{equation}
see \cite[Theorem 5.1]{Les:PDO}. The functions $e_j$
are smooth on $\R\times (\Gamma \cap \mathbb{S}^1)$ and the expansion 
\Eqref{interior} is uniform for $x$ in compact subsets of $\R$.

We choose cutoff functions $\phi$ and $\psi$, with 
$\textup{supp}\, \phi, \textup{supp}\, \psi \subset (0,1)$,
such that $\textup{supp}\, \phi \subset \textup{supp}\, \psi$ and 
$\textup{supp}\, \phi \cap \textup{supp}\, d\psi=\emptyset$. 
We define $R^I:= \psi R \phi$ and put
\begin{equation}
 \begin{split}
  \Delta(\lambda,z) R^I &= [-\partial_x^2, \psi] R \phi + \psi (\Delta^\R(\lambda,z)R-I) \phi + \phi \\
          &=: \phi + R_2(\lambda,z)\phi.
 \end{split}
\end{equation}
Note that by the choice of cutoff functions $[-\partial_x^2, \psi]$ 
and $\phi$ have disjoint support and hence 
$[-\partial_x^2, \psi] R \phi \in \CL^{-\infty}(\R;\Gamma)$.
Moreover, $\Delta^\R(\lambda,z)R-I\in \CL^{-\infty}(\R;\Gamma)$ and hence 
$R_2(\lambda,z)\in \CL^{-\infty}(\R,\Gamma)$, however the Schwartz kernel of $R_2(\lambda,z)$
is compactly supported in $(0,1)^2$. Consequently, by \Eqref{interior} we find
\begin{equation}
\label{interior-expansion}
 \begin{split}
  \Tr (\Delta(\lambda,z)^{-1} \phi) 
&= \Tr R^I - \Tr (\Delta(\lambda,z)^{-1}R_2 \phi) \\
&= \Tr (\psi \Delta^\R(\lambda,z)^{-1} \phi) + O(|(\lambda,z)|^{-\infty})\\
&\sim \sum_{j=0}^\infty e_j \!\left(\frac{(\lambda,z)}{|(\lambda,z)|}\right)
|(\lambda,z)|^{-1-j}, \ |(\lambda,z)| \to \infty.
 \end{split}
\end{equation}
This establishes a polyhomogeneous asymptotic 
expansion for the trace of the resolvent $\Delta(\lambda,z)^{-1}$ 
in the interior as a consequence of the parametric pseudo-
differential calculus. 

\subsection{The boundary parametrix}

We construct a parametrix to $\Delta(\lambda,z)$ near the boundary $x=0$.
The parametrix construction near $x=1$ works ad verbatim.

Consider $l=-\partial_x^2$ acting on $C^\infty_0(\R)$.
The operator $l$ is essentially self-adjoint in $L^2(\R)$ and we write $\bar{l}$ for its self-adjoint extension.
For $0 \leq \theta <\pi$ let $L^\theta$ be $\bar{l}$ restricted to 
\begin{equation}
\dom (L^\theta) := 
\bigsetdef{f \in H^1(\R_+)}{\cos \theta \cdot f(0) + \sin \theta\cdot f'(0) =0}.
\end{equation}
For $\mu \in \C, \Re\, \mu >0$ the resolvent kernel 
of $(L^\theta + \mu^2)^{-1}$ is given by 
\begin{align}\label{kernel-theta}
 K_\theta (x,y;\mu) = \frac{1}{2\mu} 
\left[e^{-\mu |x-y|} + C(\mu,\theta) e^{-\mu (x+y)}\right], \
C(\mu,\theta) = \frac{\mu \sin \theta + \cos \theta}
{ \mu \sin \theta - \cos \theta}.
\end{align}

The kernel $K_\R(\cdot, \cdot ;\mu)$ of the resolvent $(\bar{l}+\mu^2)^{-1}$ is given by
\begin{align}\label{2.5a}
K_\R(x,y;\mu) = \frac{1}{2\mu} \exp(-\mu |x-y|).
\end{align}

Assume below $\mu > 0$ for simplicity. Then
\begin{equation}
 \begin{split}
  \label{kernel-theta-est}
 |K_\theta (x,y;\mu)| &\leq (1+|C(\mu,\theta)|) \frac{1}{2\mu} \exp(-\mu |x-y|) \\ 
&=  (1+|C(\mu,\theta)|) K_\R(x,y;\mu).
 \end{split}
\end{equation}
We will also need an estimate for a $j$-fold convolution of the 
resolvent kernels. Let $K_\R^j(x,y;\mu)$ denote the kernel of $(\bar{l}+\mu^2)^{-j}$. 
From the formula
\begin{equation}
\frac{\partial}{\partial \mu} (\bar{l}+\mu^2)^{-j} = 2\mu(-j) (\bar{l}+\mu^2)^{-j-1},
\end{equation}
we infer 
\begin{equation}\label{j-fold}
(\bar{l}+\mu^2)^{-j} = \frac{(-1)^{j-1}}{2^{j-1} (j-1)!}
  \Bigl(\frac{1}{\mu}\frac{\partial}{\partial \mu}\Bigr)^{j-1} (\bar{l}+\mu^2)^{-1}. 
\end{equation}
From \Eqref{j-fold} and the explicit formula \Eqref{2.5a} for $K_\R$ we find
\begin{equation}
 \begin{split}
\label{kernel-real}
 K_\R^j(x,y;\mu) &= \sum_{k=0}^{j-1} \frac{1}{k!}
                   \Bigl(\prod\limits_{l=1}^{j-k-1} \frac{2l-1}{l} \Bigr)
                  \frac{|x-y|^k}{2^j \mu^{2j-1-k}}e^{-\mu|x-y|}\\ 
      & \leq \frac{1}{2} \sum_{k=0}^{j-1} \frac{1}{k!} 
         \Bigl( \frac{|x-y|}{2\mu}\Bigr)^k \frac{1}{\mu^{2j-1}} e^{-\mu|x-y|} \\
&\leq \frac{1}{2\mu^{2j-1}}\exp\left (-\frac{\mu}{2}|x-y|\right).
 \end{split}
\end{equation}

Consider smooth potentials $V,W\in C^\infty_0 (\R_+)$, 
with $V(0) > 0$ and assume for the moment that $\textup{supp}V \subset [0,\delta]$, $\delta>0$ sufficiently small, 
such that $\| V-V(0) \|_\infty \leq \frac{1}{2}V(0)$.
Abbreviate $\mu^2:= \lambda^2 V(0) + z^2$ and $\widetilde{V}:=V-V(0)$.
Consider 
\begin{align*}
(L^\theta + \lambda^2V + W + z^2)^{-1} &=
   (I + (L^\theta + \mu^2)^{-1}(\lambda^2\widetilde{V} + W))^{-1} (L^\theta +
   \mu^2)^{-1} \\
    & = \sum_{j=0}^\infty (-1)^j \left( (L^\theta + \mu^2)^{-1}(\lambda^2\widetilde{V} + W)\right)^j 
      (L^\theta + \mu^2)^{-1}.
\end{align*}
Note that the Neumann series converges in the operator norm sense, 
since for $\|\widetilde{V} \|_\infty \leq \frac{1}{2}V(0)$ and $z \gg 0$
sufficiently large we find for the operator norm
\begin{equation}\label{EqKernelEst}
 \| (L^\theta + \mu^2)^{-1}(\lambda^2\widetilde{V} + W) \| \leq
\frac{\lambda^2\|\widetilde{V}\|_\infty + \|W\|_\infty}
{\lambda^2 V(0) + z^2} <1.
\end{equation}

We also need to justify the corresponding Neumann series 
expansion for the resolvent kernel. Suppose that $V,W$ 
are both supported in $[0,\delta]$ and recall the estimates 
\Eqref{kernel-theta-est} and \Eqref{kernel-real}. 
Then for real $(\lambda,z)$ and for $0\leq x,y <\delta$
we find
\begin{align*}
&\left|\left( (L^\theta + \mu^2)^{-1}(\lambda^2\widetilde{V} + W)\right)^j (x,y) \right| \leq 
(\lambda^2\|\widetilde{V}\|_\infty + \|W\|_\infty)^j (1+|C(\mu,\theta)|)^j \\
&\times \int_\R \cdots \int_\R K_\R(x,s_1;\mu) \cdots K_\R(s_{j-1},y;\mu) \,
ds_1 \cdots ds_{j-1} \\ & = (\lambda^2\|\widetilde{V}\|_\infty + \|W\|_\infty)^j (1+|C(\mu,\theta)|)^j
K_\R^{j}(x,y;\mu) \\
& \leq \frac{\mu}{2} \Bigl(\frac{(\lambda^2\|\widetilde{V}\|_\infty + \|W\|_\infty) (1+|C(\mu,\theta)|)}
{\mu^2}\Bigr)^j \exp \left(-\frac{\mu}{2}|x-y|\right).
\end{align*}

Note that $|C(\mu,\theta)|\to 1$ as $\mu \to \infty$ and consequently 
for $\mu \geq \mu_0$ large enough, the sequence of kernels 
converges uniformly for $0\leq x,y <\delta$ and
\begin{multline}
\label{N-series-est}
\sum_{j=0}^\infty \Bigl|\bl (L^\theta + \mu^2)^{-1}(\lambda^2\widetilde{V} +
W)\br^j (x,y) \Bigr|  \\
\leq \frac{\frac 12 \mu^3 \exp (-\mu|x-y|/2)}{\mu^2 - (\lambda^2\|\widetilde{V}\|_\infty + \|W\|_\infty) (1+|C(\mu,\theta)|)}.
\end{multline}

Similar arguments also work for the derivatives of the kernels. Another convolution by $K_\R$
then yields 
\begin{prop}\label{x-derivative}
Let $V,W\in C^\infty_0(\R)$ with $V(0) > 0$ and $\| V-V(0) \|_\infty \leq \frac{1}{2}V(0)$.
Put $\mu^2=\lambda^2V(0) +z^2$. Then the kernel of $(L^\theta + \lambda^2V + W + z^2)^{-1}$
satisfies uniformly for $\mu\geq \mu_0>0$ and $\lambda \geq 0$
\begin{align}
|\partial_x^j (L^\theta + \lambda^2V + W + z^2)^{-1}(x,y)| \leq C(\mu_0)
\mu^{j-1} \exp (-\mu|x-y|/2),
\quad j=0,1.
\end{align}
\end{prop}

Consider the differential operator $\Delta_{\lambda,0}=-\partial_x^2 + \lambda^2V+W$ on 
$C^\infty_0(\R_+)$, and its self-adjoint realization $L^\theta + \lambda^2V+W$ in $L^2(\R_+)$.
We can now write down a parametrix for 
$\Delta^\theta(\lambda,z):=L^\theta + \lambda^2V+W + z^2$ near $x=0$.
We consider two cutoff functions $\phi$ and $\psi$, see Figure \ref{fig:CutOff}, both 
identically one in an open neighborhood of $x=0$ with compact 
support $\textup{supp}\, \phi, \textup{supp}\, \psi \subset [0,1)$
such that $\textup{supp}\, \phi \subset \textup{supp}\, \psi$ and 
$\textup{supp}\, \phi \cap \textup{supp}\, d\psi = \emptyset$.

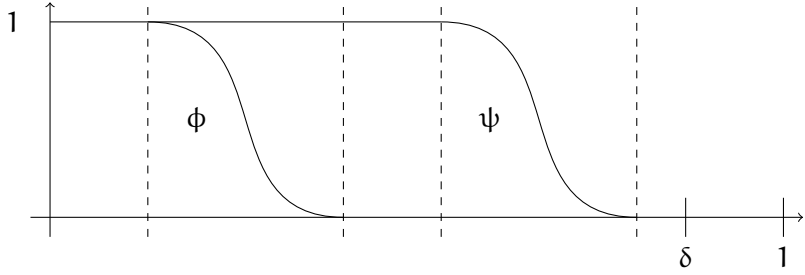
\begin{figure}[h]
\begin{center}

\begin{tikzpicture}[scale=1.3]
\draw[->] (-0.2,0) -- (7.7,0);
\draw[->] (0,-0.2) -- (0,2.2);

\draw (-0.2,2) node[anchor=east] {$1$};
\draw (6.5,-0.2) node[anchor=north] {$\delta$} -- (6.5,0.2);
\draw (7.5,-0.2) node[anchor=north] {$1$} -- (7.5,0.2);

\draw (0,2) -- (4,2);
\draw (1,2) .. controls (2.4,2) and (1.6,0) .. (3,0);
\draw[dashed] (1,-0.2) -- (1,2.2);
\draw[dashed] (3,-0.2) -- (3,2.2);
\draw (4,2) .. controls (5.4,2) and (4.6,0) .. (6,0);
\draw[dashed] (4,-0.2) -- (4,2.2);
\draw[dashed] (6,-0.2) -- (6,2.2);
\draw (1.5,1) node {$\phi$};
\draw (4.5,1) node {$\psi$};

\end{tikzpicture}

\caption{\label{fig:CutOff} The cutoff functions $\phi$ and $\psi$.}
\end{center}
\end{figure}

Given $V,W\in C^\infty[0,1]$ with $V(0) >0$ we set 
\begin{equation}\label{EqWpsiVspi}
 W_\psi=\psi W \text{ and } V_\psi=\psi V=: V(0) + \widetilde{V}_\psi,
\end{equation}
where we choose $\textup{supp} \, \psi$ 
small enough to guarantee that $\|\widetilde{V}_\psi\|_\infty \leq \frac{1}{2} V(0)$. Then we put
\begin{align*}
 R_{\partial} := \psi (L^\theta + \lambda^2V_\psi + W_\psi + z^2)^{-1} \phi.
\end{align*}
Clearly, $R_{\partial}$ maps into $\dom (L^\theta)=\dom (\Delta^\theta(\lambda,z))$. Moreover we compute 
\[
\Delta^\theta(\lambda,z)R_{\partial} = [-\partial_x^2 ,\psi] 
(L^\theta + \lambda^2V_\psi + W + z^2)^{-1} \phi + \phi
=:\phi + R_3(\lambda,z).
\]
Note that by the choice of cutoff functions, $[-\partial_x^2, \psi]$ 
and $\phi$ have disjoint support. Let $d>0$ denote the minimum of 
$|x-y|$ for $x\in \textup{supp}[-\partial_x^2, \psi]$ and 
$y\in \textup{supp}\phi$. Then there exists a constant $C>0$ such that by 
Proposition \ref{x-derivative}
\begin{align*}
\left| R_3(\lambda,z) (x,y) \right| 
\leq C \cdot 
\exp (-\mu d / 2)= O(\mu^{-\infty}), \quad \textup{as} \ \mu\to \infty.
\end{align*}
Consequently we find 
\begin{align}
 \Tr(\Delta^\theta(\lambda,z)^{-1}\phi) = \Tr R_{\partial} + O(|(\lambda,z)|^{-\infty}), 
\quad |(\lambda,z)| \to \infty,
\end{align}
and hence it suffices to establish a polyhomogeneous expansion for the 
trace of the boundary parametrix $R_{\partial}$. Write 
\begin{equation}\label{EqKplus}
 K_+ (x,y;\mu) = - \frac{1}{2\mu}C(\mu,\theta) e^{-\mu (x+y)},
\end{equation}
so that $K_\theta=K_\R+K_+$, cf. \Eqref{kernel-theta} and \Eqref{2.5a}. Moreover we abbreviate 
\begin{equation}\label{EqLaVW}
\lambda(V,W):=\lambda^2\widetilde{V}_\psi + W_\psi.
\end{equation}
Then, using $\mu^2 = \lambda^2 V(0) + z^2$, we can write
\begin{align*}
 R_{\partial} = \psi (L^\theta + \lambda^2V_\psi + W_\psi + z^2)^{-1} \phi 
= \psi (L^\theta + \mu^2 + \lambda(V,W))^{-1} \phi.
\end{align*}
 Then, by \Eqref{N-series-est} we may expand the boundary 
parametrix $R_\partial$ as a Neumann series as follows
\begin{align*}
R_\partial 
    =& \sum_{j=0}^\infty (-1)^j \psi \bigl[K_\theta \lambda(V,W)\bigr]^j 
K_\theta \phi \\
    =& \sum_{j=0}^\infty (-1)^j \psi \Bigl( \bigl[K_\theta \lambda(V,W)\bigr]^j 
          K_\theta - \bigl[K_\R \lambda(V,W)\bigr]^j K_\R \Bigr) \phi \\
     &+  \sum_{j=0}^\infty (-1)^j \psi \bigl[K_\R \lambda(V,W)\bigr]^j 
            K_\R \phi =: R^0_\partial + R^1_\partial.
\end{align*}
By similar arguments as in the previous subsection, $\Tr\bl R^1_\partial\br=
\Tr\bl R^I \br + O(\mu^{-\infty})$ 
and hence a polyhomogeneous expansion of the boundary parametrix follows 
from such an expansion of $R^0_\partial$. We write
\begin{align*}
R^0_\partial = \sum_{j=0}^{\infty} (-1)^j \psi \left( \left[K_\theta \lambda(V,W)\right]^j 
K_\theta -  \left[K_\R \lambda(V,W)\right]^j K_\R \right) \phi 
=: \sum_{j=0}^{\infty} R^{0j}_\partial.
\end{align*}

Before we proceed we note that for elements $a, b$
in a not necessarily commutative ring we have the identity
\begin{equation}\label{EqBinomId}
(a+b)^n-b^n= \sum_{j=0}^{n-1} b^j a (a+b)^{n-1-j},
\end{equation}
as one checks by induction. Consequently
\begin{equation}\label{EqRdel0j}
 \begin{split}
R^{0j}_\partial =& (-1)^j\psi \Bigl( \bigl[K_\theta \lambda(V,W)\bigr]^j 
                K_\theta -  \bigl[K_\R \lambda(V,W)\bigr]^j K_\R \Bigr) \phi \\
       =& (-1)^j \psi  \sum_{k=0}^{j-1} (K_\R \lambda(V,W))^k (K_+
       \lambda(V,W)) (K_\theta \lambda(V,W))^{j-k-1} K_\R \phi\\
       &\quad +(-1)^j \psi   (K_\theta \lambda(V,W))^{j} K_+ \phi.
   \end{split}      
\end{equation}

To expand $R^0_\partial$ we write for a fixed $M\in \N$
\begin{equation}\label{EqRsplit}
R^0_\partial = \sum_{j=0}^{M-1} R^{0j}_\partial + \sum_{j=M}^\infty
R^{0j}_\partial.
\end{equation}

The first task is to show that the trace of the second sum in \Eqref{EqRsplit}
decays sufficiently fast, more concretely $O(\mu^{-M-3/2}), \mu\to\infty$.
This is the content of the next proposition, cf. also \cite[Cor. 4.2]{Ver:MRT} 
where a parallel result is obtained for elliptic boundary value problems by a 
different method. The second task, which will occupy the whole Subsection \ref{SSPolBou}, then
is to show that the first sum in \Eqref{EqRsplit} has a polyhomogeneous
expansion. Since we may choose $M$ as large as we please we will then
obtain Proposition \ref{phg-trace}.

\begin{prop}\label{R-expansion}
Let $M\in \N, \gamma, \nu\in\N_0$ be fixed. 
For $\mu_0$ sufficiently large there exist constants $C>0, 0<q<1$ such that
for $N\ge M$ and $\mu\ge \mu_0$ 
\begin{equation}
\bigl\|\pl^\gamma_\lambda \pl_z^\nu  R^{0N}_\pl  \bigr\|_\tr 
\leq C \cdot N\cdot q^{N-M} \cdot \mu^{-M-\gamma-\nu-3/2},
\end{equation}
and consequently
\begin{equation}\label{EqRemEst}  
 \Bigl\|\pl^\gamma_\gl \pl_z^\nu \sum_{j=M}^\infty R^{0j}_\partial \Bigr\|_\tr =
 O\bl \mu^{-M-\gamma-\nu-3/2}\br, \text{ as } \mu\to\infty.
\end{equation}
Here $\|\cdot\|_\tr$ denotes the trace norm.
\end{prop}
\begin{proof} We treat the case $\gamma=\nu=0$. The case of general
 $\gamma,\nu$ follows easily since, e.g., 
 \begin{align*}
\pl_z \bl L^\theta + \mu^2 \br\ii      &= -2z  \bl L^\gt+\mu^2)^{-2},
\intertext{resp.}
   \pl_\gl \bl L^\theta + \mu^2 \br\ii &= -2 \gl \, V(0)\,  \bl
   L^\gt+\mu^2)^{-2},
  \end{align*}
and similarly for the other involved kernels.

Each of the $N$ summands of $R^{0N}_\pl$ is of the form 
\begin{equation}
P_N=\psi \, K_0\prod\limits_{j=1}^N  \gl(V,W) K_j \, \phi,
\end{equation}
where $K_j, j=0,\ldots,N$, is either $K_\R, K_+,$ or $K_\gt$.
Note that due to the factor $\gl(V,W)$ all kernels (may be assumed
to have) support $\subset [0,\delta)$. 

In view of \Eqref{EqKernelEst} we may choose $\mu_0>0$ sufficiently large such that
there exists a $0<q<1$ such that
\begin{equation}
     \|\gl(V,W) K_j\|\le q
    \end{equation}
for $\mu\ge \mu_0$ and all $j$.
Thus we may estimate the trace norm of $P_N$ by
\begin{equation}\label{EqTraNorEst}
  \bigl\| P_N \bigr\|_\tr \le \| \psi K_0 \|_\HS\cdot
  q^{N-M}\cdot \bigl\|  \prod_{j=1}^M \gl(V,W) K_j \phi \bigr\|_\HS,
\end{equation}
where $\|\cdot\|_\tr, \|\cdot\|_\HS$ denote the trace norm resp.
the Hilbert-Schmidt norm.

Of the $K_j$ in \Eqref{EqTraNorEst} at least one equals $K_+$ (cf.
\Eqref{EqRdel0j}) and by choosing those factors whose norm
we estimate by $q$ appropriately we can arrange that in \Eqref{EqTraNorEst}
at least one of the $K_j, j=1,\ldots, M$ equals $K_+$. So we have
\begin{enumerate}
\item All $K_j, j=0,\ldots, N$ satisfy the estimate 
 \begin{equation}
  |K_j(x,y;\mu)| \leq C_1 \frac{1}{\mu} e^{-\mu |x-y|}, 
 \end{equation} 
\item At least one kernel $K_j$ satisfies 
 \begin{equation}
  |K_j(x,y;\mu)| \leq C_1 \frac{1}{\mu} e^{-\mu (x+y)}.
 \end{equation} 
\end{enumerate}
For the Hilbert-Schmidt norm of $\psi K_0$ we have
\begin{equation}\label{Eq1209281}   
     \|\psi K_0 \|_\HS^2  \le \frac{C_1^2}{\mu^2} \int_0^1\int_0^1 e^{-2\mu
     |x-y|} dx dy=O(\mu^{-3}), \text{ as } \mu\to\infty.   
\end{equation}
For $Q_M:=\prod_{j=1}^M \gl(V,W) K_j$ we claim that
\begin{equation}\label{EqQMEst}  
 | Q_M(x,y;\gl,\mu)|\le \sum_{\ga,\gb} c_{\ga\gb} \cdot \mu^{-\ga}\cdot
 \max(x,y)^\gb \cdot e^{-\mu(x+y)},
\end{equation}
where the sum is over finitely many $\ga, \gb$ with the restriction
$\ga+\gb\ge M-1$. 

The Hilbert-Schmidt norm square of each summand on the right of
\Eqref{EqQMEst} can be estimated by
\begin{equation}\label{Eq1209283}   
  \begin{split}
      2 \mu^{-2\ga} &\int_0^1\int_0^y y^{2\gb} e^{-2\mu(x+y)} dx dy\\
   &\le 2 \mu^{-2\ga} \int_0^\infty y^{2\gb} e^{-2\mu y} dy\cdot
   \int_0^\infty e^{-2\mu x} dx\\
   &= O(\mu^{-2\ga-2\gb-2})=O(\mu^{-2M}), \text{ as } \mu\to\infty,
  \end{split}
\end{equation}
thus \Eqref{EqQMEst} implies
\begin{equation}\label{Eq1209284}  
 \bigl\| Q_M\bigr\|_\HS = O\bl \mu^{-M}\br,
\end{equation}
and \Eqref{EqTraNorEst}, \Eqref{Eq1209281} and \Eqref{Eq1209284} give the
claim. It therefore remains to prove \Eqref{EqQMEst}, which we single
out separately below. 
\end{proof}

\subsubsection{Proof of \Eqref{EqQMEst}} We proceed by induction on $M\in \N$.
Recall from \Eqref{EqLaVW} $\gl(V,W)=\gl^2 \tV_\psi+W_\psi,$
$\tV_\psi=\psi V-V(0)$. Recall furthermore 
from \Eqref{EqWpsiVspi} that $\psi$ was chosen such that
$\|\widetilde{V}_\psi\|_\infty \leq \frac{1}{2}V(0)$. Moreover,
since $\tV_\psi(0)=0$ and $\tV$ is smooth, we have $|\tV_\psi(x)|\le c\cdot x$
for some $c>0$. Thus
\[
   |\gl(V,W)(x)|\le c (\gl^2\cdot x+1),
\]
and hence
\[ 
|\gl(V,W) K_j(x,y;\gl,\mu)|\le c(\mu\cdot x + \mu\ii) e^{-\mu|x-y|},
\]
resp., for at least one $j$, $e^{-\mu(x+y)}$ instead of $e^{-\mu|x-y|}$. This
establishes \Eqref{EqQMEst} for $M=1$.

For the inductive step we treat the case $x\le y$. Though the kernels are not 
symmetric, the estimates for $x\ge y$ are similar. We pick one of the
summands on the right of \Eqref{EqQMEst}
\begin{align*} 
  k_1(x,y;\mu)&=\mu^{-\ga} \max(x,y)^\gb e^{-\mu(x+y)}, \quad \ga+\gb\ge M-1,\\
  \intertext{and}
  k_2(x,y;\mu)&= (\mu\cdot x + \mu\ii) e^{-\mu|x-y|}.
 \end{align*}
We split the integral $\int_0^1 k_1(x,z;\mu)k_2(z,y;\mu) dz$ into
the two parts $\int_0^y$ and $\int_y^1$. 
In the first case $z\in [0,y]$ we find
\begin{align*}
\int_0^y | &k_1(x,z;\mu) k_2(z,y;\mu) | dz \\
           &\le \mu^{-\ga} (\mu y + \mu\ii) e^{-\mu(x+y)}  y^\gb \int_0^y 1 dz \\ 
           & \le C_2 \bl\mu^{-(\ga-1)} y^{\gb+2} + \mu^{-(\ga+1)}
                      y^{\gb+1}\br e^{-\mu(x+y)};
\end{align*}
certainly $\ga-1+\gb+2\ge (M+1)-1, \ga+1+\gb+1\ge (M+1)-1$.

Secondly,
\begin{equation}\label{Eq1209286}   
  \begin{split}
   \int_y^1 | &k_1(x,z;\mu) k_2(z,y;\mu) | dz \\
         &  \le \mu^{-\ga} e^{-\mu(x-y)} \int_y^1 z^\gb (\mu z + \mu\ii)
                    e^{-2\mu z} dz\\
         & \le \mu^{-(\ga+\gb+1)} e^{-\mu(x-y)} \int_{\mu y}^\infty z^\gb (z+\mu\ii)
         e^{-2z} dz\\
         &\le C_3\bl \mu^{-(\ga+\gb+1)} + \mu^{-(\ga+\gb+2)} + \mu^{-\ga} y^{\gb+1} +
             \mu^{-(\ga+2)} y^\gb\br e^{-\mu(x+y)}.
  \end{split}
\end{equation}
In the last step we have used that for $\gamma>0$
\begin{equation}\label{eq:1209287}  
    \int_R^\infty z^\delta e^{-2z} dz \le C(\delta) (1+R^\delta) e^{-2R},\quad 0\le
    R<\infty.    
\end{equation}
The last line of \Eqref{Eq1209286} is indeed of the form as the right
hand side of \Eqref{EqQMEst} with $\ga+\gb+2\ge \ga+\gb+1\ge (M+1)-1$.
This establishes the inductive step and \Eqref{EqQMEst} is proved.
\hfill \qed

\subsection{The polyhomogeneous expansion of the boundary
parametrix}\label{SSPolBou}
It follows from Proposition \plref{R-expansion} that
a polyhomogeneous expansion of the trace of the boundary parametrix $R^0_\partial$ up
to a given order $O(\mu^{-M-3/2})$ follows from a polyhomogeneous expansion 
of the trace of the finitely many summands 
\[
\sum_{j=0}^{M-1} R^{0j}_\partial = \sum_{j=0}^{M-1} (-1)^j \psi \left( \left[K_\theta \lambda(V,W)\right]^j 
K_\theta -  \left[K_\R \lambda(V,W)\right]^j K_\R \right) \phi.
\]
Since $M$ can be chosen arbitrarily this in fact establishes a
full asymptotic expansion of the trace of $R^0_\pl$.
We establish a polyhomogeneous  expansion of the finitely many summands above 
using the microlocal formalism  of blowups. 

The kernels $K_+$ and $K_\R$ are functions on $\R^+_{1/\mu} \times (\R^+)^2_{(x,y)}$ 
with non-uniform behaviour at the diagonal $\mathscr{D}:=\{\mu=\infty, x=y\}$ and 
the highest codimension corner $\mathscr{A}:=\{\mu=\infty, x=y=0\}$. This non-uniform 
behaviour is resolved by considering an appropriate blowup $\mathscr{M}^2_b$ of $\R_+\times \R_+^2$
at $\mathscr{A}$ and $\mathscr{D}$, a procedure introduced by Melrose, see \cite{Mel:APS}, 
such that both kernels lift to polyhomogeneous distributions on the manifold with corners 
$\mathscr{M}^2_b$ in the sense of the following definition.

\begin{defn}\label{phg}
Let $X$ be a manifold with corners, with embedded boundary faces and the corresponding boundary defining functions
$\{(H_i,\rho_i)\}_{i=1}^N$. We consider distributions on $X$ that are locally restrictions of distributions 
defined across the boundaries of $X$. A distribution $\w$ on $X$ is said to be
conormal if it is of stable regularity under repeated application of vector fields on $X$ which lie tangent to all boundary faces. 
An index set $E_i = \{(\gamma,p)\} \subset {\mathbb C} \times {\mathbb N}$ 
satisfies the following hypotheses:
\begin{enumerate}
\item $\Re(\gamma)$ accumulates only at plus infinity,
\item For each $\gamma$ there is $P_{\gamma}\in \N_0$, such
that $(\gamma,p)\in E_i$ iff $p \leq P_\gamma$,
\item If $(\gamma,p) \in E_i$, then $(\gamma+j,p') \in E_i$ for all $j \in {\mathbb N}$ and $0 \leq p' \leq p$. 
\end{enumerate}
An index family $E = (E_1, \ldots, E_N)$ is an $N$-tuple of index sets. 
A conormal distribution $\w$ is polyhomogeneous on $X$ 
with index family $E$, we write $\w\in \mathscr{A}_{\textup{phg}}^E(X)$, 
if $\w$ is conormal and if in addition, near each $H_i$, 
\begin{align}\label{A}
\w \sim \sum_{(\gamma,p) \in E_i} a_{\gamma,p} \rho_i^{\gamma} (\log \rho_i)^p, \ 
\textup{as} \ \rho_i\to 0,
\end{align}
with coefficients $a_{\gamma,p}$ conormal on $H_i$, polyhomogeneous with index $E_j$
at any $H_i\cap H_j$. 
\end{defn}

There is also a space of polyhomogeneous distributions on a manifold with corners $X$ 
that are conormal to an embedded submanifold $Y\subset X$. The precise definition is given for instance in \cite{Maz:ETO}.
Morally, conormality at a submanifold models the singular behaviour of an oszillatory Fourier integral for 
some classical symbol of a prescribed order. $\mathscr{A}_{\textup{phg}}^E(X,Y)$
denotes the space of distributions conormal to $Y$, with
polyhomogeneous expansions as in \Eqref{A} at all boundary faces and 
with coefficients conormal to the intersection of $Y$ with each boundary face. 

We now continue with the definition of a blowup $\mathscr{M}^2_b$, so that 
the kernels $K_+,K_\R$ lift to polyhomogeneous distributions conormal 
to an embedded submanifold. Blowing up $\R_+\times \R_+^2$ at $\mathscr{A}$ and $\mathscr{D}$
amounts in principle to introducing polar coordinates in $\R_+\times \R_+^2$ 
at $\mathscr{A}$ and $\mathscr{D}$ together with a unique minimal differential 
structure with respect to which these coordinates are smooth. Similar construction has been employed
in \cite{Moo:HKA} and \cite{MazVer:ATM} with the difference that the blowups there
are parabolic in time direction. The resulting blowup space $\mathscr{M}^2_b$ is illustrated
in Figure \ref{blowup}.

\begin{figure}[h]
\begin{center}
\begin{tikzpicture}
\draw (0,0.7) -- (0,2);
\draw (-0.7,-0.5) -- (-2,-1);
\draw (0.7,-0.5) -- (2,-1);
\draw (0,0.7) .. controls (-0.5,0.6) and (-0.7,0) .. (-0.7,-0.5);
\draw (0,0.7) .. controls (0.5,0.6) and (0.7,0) .. (0.7,-0.5);
\draw (-0.7,-0.5) .. controls (-0.5,-0.6) and (-0.4,-0.7) .. (-0.3,-0.7);
\draw (0.7,-0.5) .. controls (0.5,-0.6) and (0.4,-0.7) .. (0.3,-0.7);
\draw (-0.3,-0.7) .. controls (-0.3,-0.3) and (0.3,-0.3) .. (0.3,-0.7);
\draw (-0.3,-1.4) .. controls (-0.3,-1) and (0.3,-1) .. (0.3,-1.4);
\draw (0.3,-0.7) -- (0.3,-1.4);
\draw (-0.3,-0.7) -- (-0.3,-1.4);
\node at (1.2,0.7) {\large{rf}};
\node at (-1.2,0.7) {\large{lf}};
\node at (1.1, -1.2) {\large{tf}};
\node at (-1.1, -1.2) {\large{tf}};
\node at (0, -1.7) {\large{td}};
\node at (0,0.1) {\large{ff}};
\draw[dashed] (0.1,-0.15) -- (0.1,2);
\draw[dashed] (0.1,-0.15) -- (-2,-0.95);
\draw[dashed] (0.1,-0.15) -- (2,-0.9);
\node at (0.6,1.9) {$\mu^{-1}$};
\node at (-2,-0.75) {$x$};
\node at (2,-0.7) {$y$};
\end{tikzpicture}
\end{center}
\caption{\label{blowup} The blowup space $\mathscr{M}^2_b$.}
\end{figure}

We make the projective coordinates on $\mathscr{M}_b^2$ explicit near the top corner and near td. 
Near the top corner of ff away from tf the projective coordinates are given by
\begin{equation}\label{top-coord}
\rho=\frac{1}{\mu}, \  \xi=\frac{x}{\rho}, \ \widetilde{\xi}=\frac{y}{\rho},
\end{equation}
where in these coordinates $\rho, \xi, \widetilde{\xi}$ are the defining functions of 
the faces ff, rf and lf respectively. The projective coordinates on $\mathscr{M}^2_b$ near the top of td away 
from tf are given by 
\begin{equation}\label{d-coord}
\eta=\frac{1}{\mu y}, \ S =\mu (x-y), y.
\end{equation}

In these coordinates tf is the face in the limit $|S|\to \infty$, ff and td are defined by 
$y, \eta$, respectively. The blowup space $\mathscr{M}^2_b$ is related to the original 
space $\R_+\times \R_+^2$ via the obvious `blow-down map'
\[
\beta: \mathscr{M}_b^2\to \R_+\times \R_+^2,
\]
which is in local coordinates simply the coordinate change back to $(1/\mu, x,y)$. 
The only difference between $\mathscr{M}^2_b$ and the heat space for incomplete
conical or edge singularities in \cite{Moo:HKA} and \cite{MazVer:ATM} 
is that here the blowup is not parabolic in $\mu^{-1}$-direction.

One can easily check in local projective coordinates above that the
kernels $K_+$ and $K_\R$ both lift to polyhomogeneous 
distributions on $\mathscr{M}^2_b$, the latter being conormal to $\beta^*\{x=y\}$. 
Put for any $k\in \N_0$ 
\begin{align}
E_k:=\{(j,0)\in \N\times \N \mid j \geq k\}.
\end{align}

Then the index set of $\beta^*K_\R$ is given by $E_1$ at ff and td, 
by $E_0$ at rf and lf. The index sets of $K_+$ are the same at ff, rf and lf, 
and given by $E_{\infty}$ at tf, i.e. $\beta^*K_+$ is vanishing to infinite order at the temporal face tf. 

We denote by $\mathscr{A}_{\textup{phg}}^{l,p,E_{\textup{lf}}, 
E_{\textup{rf}}}(\mathscr{M}^2_b, \beta^*\{x=y\})$ 
the space of polyhomogeneous distributions on
$\mathscr{M}^2_b$ conormal up to $\beta^*\{x=y\}$, with index set
$E_l, l\in \N$ at ff, the index set $E_p, p\in \N$ at td, index sets $(E_{\textup{lf}}, E_{\textup{rf}})$ 
at lf and rf, respectively, and vanishing to infinite order at td. 
The space $\mathscr{A}_{\textup{phg}}^{l,p,E_{\textup{lf}}, 
E_{\textup{rf}}}(\mathscr{M}^2_b)$ denotes the subspace 
of polyhomogeneous distributions that are smooth across $\beta^*\{x=y\}$. 

Clearly, $K_\R \in \mathscr{A}_{\textup{phg}}^{1,1,E_0, E_0}(\mathscr{M}^2_b,\beta^*\{x=y\})$, 
as the Schwartz kernel of $(\bar{l} +\mu^2)^{-1}$ inside the 
strongly parametric calculus. Moreover $K_+ \in \mathscr{A}_{\textup{phg}}^{1,\infty ,E_0, E_0}
(\mathscr{M}^2_b)$. Polyhomogeneity for the various compositions 
of $K_+$ and $K_\R$, or more generally of any 
$K_A \in \mathscr{A}_{\textup{phg}}^{l,p,E_{\textup{lf}}, E_{\textup{rf}}}(\mathscr{M}^2_b)$ 
and $K_B \in \mathscr{A}_{\textup{phg}}^{l',\infty,E'_{\textup{lf}}, E'_{\textup{rf}}}(\mathscr{M}^2_b)$
is asserted by in the following

\begin{prop} \cite[Prop. 3.2]{Ver:MRT}
\label{composition}
For index sets $E_{\lf}$ and $E'_{\rf}$ such that $E_{\lf}+E'_{\rf}>-1$, we have
\[
\mathscr{A}_{\textup{phg}}^{l,p,E_{\textup{lf}}, E_{\textup{rf}}}(\mathscr{M}^2_b) \circ 
\mathscr{A}_{\textup{phg}}^{l',\infty,E'_{\textup{lf}}, E'_{\textup{rf}}}(\mathscr{M}^2_b,\beta^*\{x=y\}) \subset 
\mathscr{A}_{\textup{phg}}^{l+l'+1,\infty, E'_{\textup{lf}}, E_{\textup{rf}}}(\mathscr{M}^2_b).
\]
\end{prop}

This composition result is proved on the basis of the Pushforwrad theorem by Melrose \cite{Mel:APS}
using the so called triple space construction. Asymptotic expansion for the trace of each polyhomogeneous
kernel in the space $\mathscr{A}_{\textup{phg}}^{l,\infty,E_{\textup{lf}}, E_{\textup{rf}}}(\mathscr{M}^2_b)$
is a simple consequence of the Pushforward theorem as well.

\begin{prop} \cite[Prop. 4.3]{Ver:MRT}
\label{trace-boundary}
For any $K \in 
\mathscr{A}_{\textup{phg}}^{l,\infty,E_{\textup{lf}}, E_{\textup{rf}}}(\mathscr{M}^2_b)$
and any cutoff function $\phi \in C^\infty_0(\R_+)$ with $\phi \equiv 1$ 
in an open neighborhood of zero, we find for $(E_{\textup{lf}} + E_{\textup{rf}} + 1) >-1$ that
\begin{align}
\Tr (K\phi) = \int_0^\infty K(x,x; \mu)\phi(x) dx \sim \sum_{j=0}^\infty 
a_j \mu^{-(l+1)-j}, \ \mu \to \infty.
\end{align} 
\end{prop}

We can now employ Proposition \ref{composition} together with Proposition \ref{trace-boundary}
in order to derive a polyhomogeneous expansion of the finite sum (recall 
$\lambda(V,W)=\lambda^2\widetilde{V}_\psi + W_\psi$)
\[
\sum_{j=0}^{M-1} (-1)^j \psi \left( \left[K_\theta \lambda(V,W)\right]^j 
K_\theta -  \left[K_\R \lambda(V,W)\right]^j K_\R \right) \phi = \sum_{j=0}^{M-1}R^{0j}_{\partial}. 
\]
Each $R^{0j}_{\partial}$ is finite number of 
summands of the form $K(j,p), j\leq (M-1)$, which are given by a convolution of 
$(j+1)$ kernels $K_\R $ and $K_+$, with at least one $K_+$ and  
$p(\leq j)$ times $\lambda^2 \widetilde{V}_\psi$. Note that 
$\widetilde{V}_\psi(x)=O(x)$ as $x\to 0$, and is smooth so that 
\begin{align*}
K_\R\widetilde{V}_\psi\in  \mathscr{A}_{\textup{phg}}^{2,1,E_1, E_0}(\mathscr{M}^2_b, \beta^*\{x=y\}),
\quad K_+\widetilde{V}_\psi\in  \mathscr{A}_{\textup{phg}}^{2,\infty,E_1, E_0}(\mathscr{M}^2_b). 
\end{align*}
Consequently we find by Proposition \ref{composition}
\[
K(j,p) \in \lambda^{2p} \mathscr{A}_{\textup{phg}}^{2j+1+p,\infty,E_0, 
E_0}(\mathscr{M}^2_b).
\]
Proposition \ref{trace-boundary} now implies 
\[
\Tr K(j,p) \sim \sum_{i=0}^\infty a_i \frac{\lambda^{2p}}
{\mu^{2(j+1)+p+i}} =: \sum_{i=0}^\infty a_i^{jp} (\lambda, z),
\]
where each $a_i^{jp}$ is homogeneous in $(\lambda, z)$ of homogeneity 
degree $(p-2(j+1)-i)$. Consequently, overall we obtain
\begin{equation}
\label{boundary-expansion}
\Tr \sum_{j=0}^{M-1}R^{0j}_{\partial} 
\sim \sum_{i=0}^\infty e_i(\lambda,z), \ |(\lambda,z)| \to \infty,
\end{equation}
where each $e_i\in C^{\infty}(\R^2_+\setminus \{(0,0)\})$ is homogeneous
of order $(-2-i)$ jointly in both variables. We have now all ingredients
to prove Proposition \ref{phg-trace}.

\subsection{Proof of Proposition \ref{phg-trace}}\label{sub-int}
This is now a consequence of the interior expansion \Eqref{interior-expansion}, 
Proposition \plref{R-expansion} and \Eqref{boundary-expansion}. 
Comparing \Eqref{R-expansion} and \Eqref{boundary-expansion} we see that
the leading term in the polyhomogeneous expansion of $\Tr(\Delta_\lambda+z^2)^{-1}$ indeed 
comes from the interior. \hfill\qed

\subsection{Proof of Theorem \ref{trace-sum}}
As an application of the polyhomogeneity of the resolvent trace 
we now prove Theorem \plref{trace-sum} and 
clarify in which sense the trace expansions 
of $\Tr(\Delta_\lambda+z^2)^{-2}$ sum up to the trace expansion 
of $\Tr(\Delta+z^2)^{-2}$. Note that $\Tr(\Delta_\lambda+z^2)^{-2}=O(z^{-3})$, 
whereas $\Tr(\Delta+z^2)^{-2}=O(z^{-2})$, as $z\to \infty$. 
So the trace expansions clearly do not sum up in an obvious way.

We first note that Proposition \ref{phg-trace} implies
\begin{equation}\label{ML-1}
\Tr(\Delta_\lambda+z^2)^{-2} = (2z)^{-1} \partial_z
\Tr(\Delta_\lambda+z^2)^{-1} \sim \sum_{i=0}^\infty h_i (\lambda,z), \ |(\lambda,z)| \to \infty,
\end{equation}
where $\gamma_i:=(i+3)$ and each $h_i\in C^{\infty}(\R^2_+\setminus \{(0,0)\})$
is homogeneous of order $(-\gamma_i)$ jointly in both variables. 
In particular for fixed $z$ we have $\Tr(\Delta_\gl+z^2)^{-2}=O(\gl^{-3}),
\gl\to\infty$. Hence the sum 
$\sum_{\lambda=1}^{\infty} \Tr(\Delta_\lambda+z^2)^{-2}$
and the integral $\int_1^\infty \Tr(\Delta_\lambda+z^2)^{-2} d\lambda $
converge.

We apply the Euler MacLaurin formula \Eqref{EqRegsum2} to
$f(\gl)=\Tr(\Delta_\gl+z^2)^{-2}$ and find for $M\in \N$ sufficiently large
\begin{align}
\sum_{\lambda=1}^{\infty} \Tr(\Delta_\lambda+z^2)^{-2}
   =&\int_1^\infty \Tr(\Delta_\lambda+z^2)^{-2} d\lambda +
   \frac{1}{2}\Tr(\Delta_1+z^2)^{-2}\nonumber \\
    &    - \sum_{k=1}^M \frac{B_{2k}}{(2k)!} \, 
\pl_\lambda^{(2k-1)}\Tr(\Delta_\lambda+z^2)^{-2}|_{\lambda=1} \label{EM2} \\ 
&+ \frac{1}{(2M+1)!} \int_1^\infty B_{2M+1}(\lambda-[\lambda]) \,
\pl_\lambda^{(2M+1)}\Tr(\Delta_\lambda+z^2)^{-2} d\lambda.\nonumber
\end{align}
We need to establish the asymptotic behaviour of each of the terms above as $z\to \infty$.
The standard resolvent trace expansion, cf. \Eqref{trace-2-expansion}
yields
\begin{equation}\label{exp1}
\begin{split}
\frac{1}{2}\Tr(\Delta_1+z^2)^{-2}  - \sum_{k=1}^M \frac{B_{2k}}{(2k)!} \, 
\partial_\lambda^{(2k-1)}\Tr(\Delta_\lambda+z^2)^{-2}|_{\lambda=1} 
\sim \sum_{i=0}^{\infty} a_i z^{-3-i}, \ z\to \infty.
\end{split}
\end{equation}
Moreover, Proposition \ref{phg-trace} implies
\begin{equation} \label{exp2}
\begin{split}
\frac{1}{(2M+1)!} \int_1^\infty \mid B_{2M+1}(\lambda-[\lambda]) \, 
\partial_\lambda^{(2M+1)}\Tr(\Delta_\lambda+z^2)^{-2} \mid d\lambda\\
\leq C \cdot \int_1^\infty (\lambda+z)^{-4-2M} d\lambda =O(z^{-3-2M}), \ z\to\infty.
\end{split}
\end{equation}
The asymptotic expansion of the first integral term in \Eqref{EM2}
now follows from \Eqref{ML-1} and 
\begin{displaymath}
\int_1^\infty h_i(\lambda, z) \, d\lambda = 
z^{-\gamma_i} \int_1^\infty h_i(\lambda / z, 1) \, d\lambda =
z^{-\gamma_i + 1} \int_1^\infty h_i(\nu, 1) \, d\nu.
\end{displaymath}
The $\nu$-integral is finite, since as a consequence of 
smoothness of $h_i(1,\cdot)$ at $z=0$ and homogeneity, 
$h_i(\nu,1) = O(\nu^{-\gamma_i}), \gamma_i\geq 3$.

The discussion of $\sum\limits_{\gl=-\infty}^0 \Tr(\Delta_\gl+z^2)^{-2}$
is similar and the proof of Theorem \plref{trace-sum} is complete.
\hfill\qed

\section{Proof of the Fubini Theorem \ref{fubini}}

\subsection{Double integrals and regularized limits of homogeneous functions}
\label{SSSDIRL}
\newcommand{\quplane}{\R_+^2\setminus \{(0,0)\}}
As an application of the regularized limit and as a preparation to the proof
of the Fubini Theorem \ref{fubini} we discuss regularized iterated integrals 
of homogeneous functions on the quarter plane $\R_+^2$. 
During this section let $f\in C^\infty(\R^2_+\setminus \{(0,0)\})$
be homogeneous of order $\ga\in\C$. That is for $(x,y)\in \quplane, \gl>0$
we have $f(\gl\cdot x,\gl \cdot y)=\gl\cdot f(x,y)$. Taylor expansion about
$(1,0), (0,1)$ yields the expansions
\begin{align}
    f(x,y) & = y^\ga f(x/y,1) \sim \sum_{j=0}^\infty c_j \, y^{\ga -j } x^j,
    \quad y\to\infty, x\le x_0\label{EqTaylor1}\\
    f(x,y) & = x^\ga f(1,y/x) \sim \sum_{j=0}^\infty d_j \, x^{\ga -j } y^j,
   \quad x\to\infty, y\le y_0\label{EqTaylor2}.
\end{align}
If $\ga+1 \not\in\Z_+$ we put $c_{\ga+1}=d_{\ga+1}=0$ such that these
coefficients are always defined. They will play a crucial role.

\begin{lemma}\label{LLimits} Let $f\in C^\infty(\R^2_+\setminus \{(0,0)\})$
be homogeneous of order $\ga\in\C$.
\par\upshape{1. } Let $\ga+2\not=0, b>0$. Then
\begin{align*}
 \LIM_{z\to\infty} z^{\ga+2} \regint_{b/z}^\infty f(1,y) dy &=
 \frac{-d_{\ga+1}}{\ga+2} b^{\ga+2},\\
 \LIM_{z\to 0} z^{\ga+2} \regint_{b/z}^\infty f(1,y) dy &=0
\end{align*}

\par\upshape{2. } For $\ga+2=0, b>0$ the $\LIM$ of
     $ \log z\, \regint_{b/z}^\infty f(1,y) dy$
as $z\to 0$ and as $z\to \infty$ vanishes.
\end{lemma}
\begin{proof}
 \Eqref{EqTaylor1} implies that 
\[
\regint_z^\infty f(1,y) dy \sim_{z\to\infty} \sum_{j=0, j\not=\ga+1}^\infty
   \frac{-c_j}{\ga-j+1} z^{\ga-j+1} - c_{\ga+1} \log z
\]
and \Eqref{EqTaylor2} implies that
\[
\regint_z^\infty f(1,y) dy \sim_{z\to 0} \regint_0^\infty f(1,y) dy - \sum_{j=0}^\infty
   \frac{d_j}{j+1} z^{j+1}.
\]
From these two asymptotic expansions the Lemma follows immediately.
\end{proof}

\begin{lemma}\label{LHomDoubleInt} Let $f\in C^\infty(\R^2_+\setminus \{(0,0)\})$
be homogeneous of order $\ga\in\C$. 
 \par\upshape{1. } If $\ga+2\not=0, a,b\ge 0, a+b>0$ then we have
\begin{equation}
 \begin{split}\label{EqDoubleInt1}
 \regint_a^\infty &\regint_b^\infty f(x,y) dy dx  = 
 \regint_b^\infty \regint_a^\infty f(x,y) dx dy   \\
 =& -\frac {a^{\ga+2}}{\ga+2} \regint_{b/a}^\infty f(1,y) dy - 
     \frac {b^{\ga+2}}{\ga+2} \regint_{a/b}^\infty f(x,1) dy \\
  &  -c_{\ga+1} \regint_a^\infty x^{\ga+1} \log x dx
     -d_{\ga+1} \regint_b^\infty x^{\ga+1} \log x dx.
   \end{split}
  \end{equation}   
  \par\upshape{2. } If $\ga=-2$ then we have
\begin{equation}
 \begin{split}\label{EqDoubleInt2}
 \regint_a^\infty \regint_b^\infty f(x,y) dy dx  
 =& -\log a \cdot \regint_{b/a}^\infty f(1,y) dy - 
    \log b \cdot  \regint_{a/b}^\infty f(x,1) dy \\
    &  -\regint_{a/b}^\infty f(x,1) \log x dx. 
   \end{split}
 \end{equation}   
If $a=0$ or $b=0$ then in the right hand side one has to take the
appropriate regularized limit. More concretely, one has e.g.
\begin{align}
 &\LIM_{a\to 0} a^{\ga+2} \regint_{b/a}^\infty f(1,y) dy =0, & \ga+2\not
 =0,\label{EqLimits1}\\
 &\LIM_{a\to 0} \log a \cdot \regint_{b/a}^\infty f(1,y) dy =0, & \ga=-2.
 \label{EqLimits2}
\end{align}

If $a=b=0$ then regardless of the value of $\ga$ the iterated
integral $\regint_0^\infty\regint_0^\infty f $ vanishes for
both orders of integration.
\end{lemma}
\begin{proof}
The existence of the regularized integrals 
$\regint_c^\infty f(x,1)dx, \regint_c^\infty f(1,y)dy$ follows
from the expansions \Eqref{EqTaylor1}, \eqref{EqTaylor2}.
The proof of \Eqref{EqDoubleInt1} and \Eqref{EqDoubleInt2}
is a straightforward exercise in integration. We will mention the steps
where caution due to the regularization process is needed. 

\subsubsection*{$\ga+2\not=0, b>0$} We look
at the integral $\regint_a^\infty\regint_b^\infty f(x,y)dy dx$. In the
inner integral substitute $y\to x\cdot y$ and apply Lemma \ref{LChangeVar}
to obtain
\begin{equation}\label{EqProofFub0}
\regint_a^\infty x^{\ga+1} \regint_{b/x}^\infty f(1,y) dy dx - c_{\ga+1}
 \regint_a^\infty x^{\ga+1} \log x dx,
\end{equation}
where \Eqref{EqTaylor1} was used. In the first summand integrate by parts
to obtain
\begin{equation}\label{EqProofFub1}
\LIM_{R\to\infty} \frac{x^{\ga+2}}{\ga+2} \int_{b/x}^\infty f(1,y) dy
\bigg\vert_{x=a}^{x=R}- \regint_a^\infty \frac{b^{\ga+1}}{\ga+2} f(x/b,1)dx.
\end{equation}
By Lemma \ref{LLimits} the first $\LIM$ as $R\to\infty$ term
equals
\[\begin{cases}
-\frac{a^{\ga+2}}{\ga+2} \regint_{b/a}^\infty f(1,y) dy -
\frac{d_{\ga+1}}{(\ga+2)^2} b^{\ga+2}, & a>0,\\
  -\frac{d_{\ga+1}}{(\ga+2)^2} b^{\ga+2}, & a=0.
 \end{cases}  
\]
In the second integral in \Eqref{EqProofFub1} we use the homogeneity of $f$,
substitute $x\to b\cdot x$ and apply Lemma  \ref{LChangeVar} and
\Eqref{EqTaylor2} to obtain
\[
\frac {b^{\ga+2}}{\ga+2} \regint_{a/b}^\infty f(x,1) dx 
     -\frac{d_{\ga+1}}{\ga+2} b^{\ga+2} \log b.
\]
Taking into account \Eqref{EqExRegInt2} we obtain 1. of the Lemma.

\subsubsection*{$\ga=-2, b>0$} This is essentially the same calculation;
integration by parts in \Eqref{EqProofFub0} now yields
\begin{equation}\label{EqProofFub1a}
\LIM_{R\to\infty} \log x\cdot  \int_{b/x}^\infty f(1,y) dy
\bigg\vert_{x=a}^{x=R}- \regint_a^\infty \log x \cdot b^{-1} f(x/b,1)dx.
\end{equation}
Now one proceeds as before.

\subsubsection*{$b=0$} This case is even simpler. \Eqref{EqProofFub0} now reads
\begin{equation}\label{EqProofFub0a}
\regint_a^\infty x^{\ga+1} \regint_{0}^\infty f(1,y) dy dx - c_{\ga+1}
 \regint_a^\infty x^{\ga+1} \log x dx,
\end{equation}
and the remaining claims follow easily from Lemma \ref{LLimits}.
\end{proof}

Finally, we need the following result about interchanging $\LIM$ and
$\regint$ resp. $\regint$ and differentiation.

\begin{lemma}\label{LLimitExchange}
Let $f:\quplane\to\C$ be smooth and $\ga$--homogeneous. Then we have
for $a\ge 0$ 
\begin{align}
    \LIM_{y\to\infty} \regint_a^\infty f(x,y) dx
              &= \regint_a^\infty \LIM_{y\to\infty} f(x,y) dy + \textup{Corr},
              \label{EqLimitExchange1}\\
    \frac{d}{dy} \regint_a^\infty f(x,y) dx & =           
          \regint_a^\infty \frac{\pl}{\pl y} f(x,y) dx,
              \label{EqLimitExchange2}
\end{align}   
where $\textup{Corr} = \regint_0^\infty f(x,1) dx$ if $\ga=-1$ and zero otherwise.        
\end{lemma}
\begin{proof}
 From \Eqref{EqTaylor1} we infer
\[
\regint_a^\infty \LIM_{y\to\infty} f(x,y) dx = c_\ga \regint_a^\infty x^\ga dx
= -\frac{c_\ga}{\ga+1}a^{\ga+1},
\]
with the understanding that the right hand side is $0$ for $\ga=-1$ since $c_{-1}=0$.
This expression is set to zero if $a=0$.
On the other hand substituting $x\to y\cdot x$ we find using Lemma
\plref{LChangeVar}
\[
\regint_a^\infty f(x,y) dx = \regint_{a/y}^\infty y^{\ga+1} f(x,1) dx -
d_{\ga+1} y^{\ga+1} \log y
\]
and taking the regularized limit as $y\to \infty$ yields by Lemma \ref{LLimits} again 
$-\frac{c_\ga}{\ga+1}a^{\ga+1}$ if $\ga\neq -1$, with the understanding that this expression 
is set to zero of $a=0$. In case $\ga=-1$ Lemma \ref{LLimits} does not provide any statement, however
the regularized limit is straightforward and equals $\regint_0^\infty f(x,1)dx$. The first formula is proved.

Similarly, writing 
\[
f(x,y) \sim_{x\to\infty} \sum_{j=0}^M d_j x^{\ga -j} y^j + R_M(x,y)
\]
with $R_M$ and its $y$-derivative being  $O(x^{\ga - M})$ locally
uniformly in $y$ we conclude that for $R_M$ 
differentiation by $y$ and integration integration by $x$ can be interchanged.
Furthermore, for each summand we have
\[
 \frac{d}{dy} \regint_a^\infty x^{\ga-j} y^j dx 
      = \regint_a^\infty x^{\ga-j} \frac{d}{dy} y^j dx
      = \regint_a^\infty x^{\ga-j} dx \cdot j y^{j-1}
\]
and the proof is complete.
\end{proof}

\subsection{Proof of the Fubini Theorem \ref{fubini}} \label{SSFubini}

After the preparations done in the previous Subsection the proof of 
the Fubini Theorem \ref{fubini} is now straightforward.

Choosing $N$ large enough it suffices to prove
the Fubini Theorem \ref{fubini} for a homogeneous function
$f_{\A}\in C^{\infty}(\R_+^2 \setminus \{(0,0)\})$, where $\A$ denotes the
degree of homogeneity. For homogeneous functions \Eqref{fubini-eq} is
a direct consequence of Lemma \ref{LHomDoubleInt}.
 
To see \Eqref{EqFubiniSum} for homogeneous $f$ 
we first employ the Euler MacLaurin formula
\Eqref{EqRegsum2} and find for $M$ large enough
\begin{align}
 \regint_a^\infty \regsum_{\lambda=N}^\infty f(x,\lambda)   dx 
    =& \regint_a^\infty \regint_N^\infty f(x,\gl) d\gl dx \nonumber \\
     &+ \sum_{k=1}^M \frac{B_{2k}}{(2k)!} 
    \regint_a^\infty \LIM_{R\to \infty} \pl_2^{2k-1}f(x,R) - \pl_2^{2k-1} f(x,N) dx 
       \label{EqRegsum2a}\\ 
     &+\frac{1}{2} \regint_a^\infty f(x,N) + \LIM_{R\to\infty} f(x,R)dx \\
     & + \frac{1}{(2M+1)!} \int_a^\infty\int_1^\infty 
     B_{2M+1}(\gl-[\gl])\pl_2^{2M+1}f(x,\gl)d\gl dx . \nonumber 
\end{align}
In each term on the right we need to interchange $\regint_a^\infty \ldots dx$
with a limit process. To the first summand we apply \Eqref{fubini-eq}
and obtain the correction term $\int_0^\infty f_{-2}(x,1)\log x\, dx$.
To the second and third summand we apply Lemma \ref{LLimitExchange}
to exchange $\regint_a^\infty$ and $\LIM$ resp. $\pl_2$.
Finally, if $\pl_2^{2M+1}$ is homogeneous of degree $\ga-2M-1$
such that for $M$ large enough the double integral in the last
summand converges in the Lebesgue sense and therefore the order
of integration may be exchanged. \Eqref{EqFubiniSum} is proved.
\hfill \qed

\subsection{Proof of Theorem \ref{main-thm}}
We choose $N$ large enough such that $\Delta_\gl$ is invertible for 
$|\gl|\ge N$. Then
\begin{equation}
\label{ABC}
\begin{split}
\log \det\nolimits_{\zeta} \Delta 
=& -2 \regint_0^\infty z^3 \Tr(\Delta + z^2)^{-2} dz \\
=& - 4 \regint_0^\infty z^3 \sum_{\lambda=1}^\infty \Tr(\Delta_\lambda + z^2)^{-2} dz 
- 2  \regint_0^\infty z^3 \Tr(\Delta_0 + z^2)^{-2} dz \\
=& \log \det\nolimits_{\zeta} \Delta_0 - 
   4 \sum_{\gl=1}^{N-1} \regint_1^\infty z^3 \Tr(\Delta_\lambda + z^2)^{-2} dz
   \\
  & - 4 \regint_0^\infty z^3 \sum_{\lambda=N}^\infty \Tr(\Delta_\lambda +
  z^2)^{-2} dz.
\end{split}
\end{equation}
By Proposition \plref{phg-trace} we may apply the
Fubini Theorem \Eqref{EqFubiniSum} to the last sum. Interchanging
sum and integration yields the correction terms as in Theorem \ref{fubini} and the Theorem is proved. 
\hfill\qed

\section*{Acknowledgements}
The authors gratefully acknowledge helpful discussions with Leonid Friedlander and Rafe Mazzeo. 
They also thank Benedikt Sauer for a careful reading of the manuscript.
The authors would like to thank the anonymous referee for valuable comments and suggestions.
Both authors were supported by the Hausdorff Center for Mathematics.

\bibliography{mlbib,localbib}

\def\cprime{$'$}
\providecommand{\bysame}{\leavevmode\hbox to3em{\hrulefill}\thinspace}
\providecommand{\MR}{\relax\ifhmode\unskip\space\fi MR }
\providecommand{\MRhref}[2]{%
  \href{http://www.ams.org/mathscinet-getitem?mr=#1}{#2}
}
\providecommand{\href}[2]{#2}
\begin{thebibliography}{\textsc{MaVe11}}

\bibitem[\textsc{BKD96}]{BKD:HKA}
\textsc{M.~Bordag}, \textsc{K.~Kirsten}, and \textsc{S.~Dowker},
  \emph{Heat-kernels and functional determinants on the generalized cone},
  Comm. Math. Phys. \textbf{182} (1996), no.~2, 371--393. \MR{1447298
  (98d:58193)}

\bibitem[\textsc{ChCo12}]{ChaCon:SAR}
\textsc{A.~H. Chamseddine} and \textsc{A.~Connes}, \emph{Spectral action for
  {R}obertson-{W}alker metrics}, J. High Energy Phys. (2012), no.~10, 101,
  front matter + 29. \MR{3033848}

\bibitem[\textsc{EgSc97}]{EgSh:PDO}
\textsc{Y.~V. Egorov} and \textsc{B.-W. Schulze}, \emph{Pseudo-differential
  operators, singularities, applications}, Operator Theory: Advances and
  Applications, vol.~93, Birkh\"auser Verlag, Basel, 1997. \MR{1443430
  (98e:35181)}

\bibitem[\textsc{Gil95}]{Gil:ITH}
\textsc{P.~B. Gilkey}, \emph{Invariance theory, the heat equation, and the
  {A}tiyah-{S}inger index theorem}, second ed., Studies in Advanced
  Mathematics, CRC Press, Boca Raton, FL, 1995. \MR{1396308 (98b:58156)}

\bibitem[\textsc{Gru96}]{Gru:FCO}
\textsc{G.~Grubb}, \emph{Functional calculus of pseudodifferential boundary
  problems}, second ed., Progress in Mathematics, vol.~65, Birkh\"auser Boston,
  Inc., Boston, MA, 1996. \MR{1385196 (96m:35001)}

\bibitem[\textsc{GSW06}]{GSW:EFS}
\textsc{V.~W. Guillemin}, \textsc{S.~Sternberg}, and \textsc{J.~Weitsman},
  \emph{The {E}hrhart function for symbols}, Surveys in differential geometry.
  {V}ol. {X}, Surv. Differ. Geom., vol.~10, Int. Press, Somerville, MA, 2006,
  pp.~31--41. \MR{2408221 (2009k:52028)}

\bibitem[\textsc{Les97}]{Les:OFT}
\textsc{M.~Lesch}, \emph{Operators of {F}uchs type, conical singularities, and
  asymptotic methods}, Teubner-Texte zur Mathematik [Teubner Texts in
  Mathematics], vol. 136, B. G. Teubner Verlagsgesellschaft mbH, Stuttgart,
  1997. \texttt{arXiv:dg-ga/9607005v1}, \MR{1449639 (98d:58174)}

\bibitem[\textsc{Les98}]{Les:DOR}
\bysame, \emph{Determinants of regular singular {S}turm-{L}iouville operators},
  Math. Nachr. \textbf{194} (1998), 139--170. \MR{1653090 (99j:58220)}

\bibitem[\textsc{Les10}]{Les:PDO}
\bysame, \emph{Pseudodifferential operators and regularized traces}, Motives,
  quantum field theory, and pseudodifferential operators, Clay Math. Proc.,
  vol.~12, Amer. Math. Soc., Providence, RI, 2010, pp.~37--72.
  \texttt{arXiv:0901.1689 [math.OA]}, \MR{2762524}

\bibitem[\textsc{MaVe11}]{MazVer:ATM}
\textsc{R.~Mazzeo} and \textsc{B.~Vertman}, \emph{Analytic torsion on manifolds
  with edges},  \texttt{arXiv:1103.0448v1 [math.SP]}.

\bibitem[\textsc{Maz91}]{Maz:ETO}
\textsc{R.~Mazzeo}, \emph{Elliptic theory of differential edge operators. {I}},
  Comm. Partial Differential Equations \textbf{16} (1991), no.~10, 1615--1664.
  \MR{1133743 (93d:58152)}

\bibitem[\textsc{Mel93}]{Mel:APS}
\textsc{R.~B. Melrose}, \emph{The {A}tiyah-{P}atodi-{S}inger index theorem},
  Research Notes in Mathematics, vol.~4, A K Peters Ltd., Wellesley, MA, 1993.
  \MR{1348401 (96g:58180)}

\bibitem[\textsc{Moo99}]{Moo:HKA}
\textsc{E.~A. Mooers}, \emph{Heat kernel asymptotics on manifolds with conic
  singularities}, J. Anal. Math. \textbf{78} (1999), 1--36. \MR{1714065
  (2000g:58039)}

\bibitem[\textsc{Sau13}]{Sau:ORT}
\textsc{B.~Sauer}, \emph{On the resolvent trace of multi-parametric
  {S}turm-{L}iouville operators}, Diplomarbeit, Universit\"at Bonn, 2013.

\bibitem[\textsc{See69}]{See:TRO}
\textsc{R.~Seeley}, \emph{The resolvent of an elliptic boundary problem}, Amer.
  J. Math. \textbf{91} (1969), 889--920. \MR{0265764 (42 \#673)}

\bibitem[\textsc{Shu01}]{Shu:POS}
\textsc{M.~A. Shubin}, \emph{Pseudodifferential operators and spectral theory},
  second ed., Springer-Verlag, Berlin, 2001, Translated from the 1978 Russian
  original by Stig I. Andersson. \MR{1852334 (2002d:47073)}

\bibitem[\textsc{Spr05}]{Spr:ZFA}
\textsc{M.~Spreafico}, \emph{Zeta function and regularized determinant on a
  disc and on a cone}, J. Geom. Phys. \textbf{54} (2005), no.~3, 355--371.
  \MR{2139088 (2005k:11184)}

\bibitem[\textsc{Spr06}]{Spr:ZIF}
\textsc{M.~Spreafico}, \emph{Zeta invariants for {D}irichlet series}, Pacific
  J. Math. \textbf{224} (2006), no.~1, 185--200. \MR{2231657 (2007b:11129)}

\bibitem[\textsc{Ver09}]{Ver:ATO}
\textsc{B.~Vertman}, \emph{Analytic torsion of a bounded generalized cone},
  Comm. Math. Phys. \textbf{290} (2009), no.~3, 813--860. \MR{2525641
  (2010d:58032)}

\bibitem[\textsc{Ver13}]{Ver:MRT}
\bysame, \emph{Multiparameter resolvent trace expansion for elliptic boundary
  problems},  \texttt{arXiv:1301.7293 [math.SP]}.

\end{thebibliography}
\bibliographystyle{amsalpha-lmp}
\listoffigures

\end{document}